\documentclass[12pt,a4paper]{article}
\usepackage{mathrsfs}
\usepackage{epsfig, graphicx}
\usepackage{latexsym,amsfonts,amsbsy,amssymb}
\usepackage{amsmath,amsthm}
\usepackage{color}
\usepackage{hyperref,color}
\makeatletter 

\@addtoreset{equation}{section}
\makeatother 
\allowdisplaybreaks 

\newtheorem{theorem}{Theorem}[section]
\newtheorem{lemma}{Lemma}[section]
\newtheorem{corollary}{Corollary}[section]

\newtheorem{example}{Example}[section]
\newcommand{\bfp}{\mathbf{p}}
\newcommand{\bfq}{\mathbf{q}}
\newcommand{\bfW}{\mathbf{W}}
\newcommand{\bfw}{\mathbf{w}}

\newcommand{\dive}{\mathrm{div}}
\begin{document}
\title{Computable Error Estimates for Ground State Solution of
Bose-Einstein Condensates
\footnote{This work is supported in part National Science Foundations of China
(NSFC 91330202, 11371026, 11001259, 11031006, 2011CB309703)
and the National Center for Mathematics and Interdisciplinary Science, CAS.}
}
\author{
Hehu Xie\footnote{LSEC, ICMSEC, Academy of Mathematics and Systems Science, Chinese Academy of
Sciences, Beijing 100190, China (hhxie@lsec.cc.ac.cn)} \ \ \ {\rm and} \ \ \
Manting Xie\footnote{LSEC, ICMSEC, Academy of Mathematics and Systems Science, Chinese Academy of
Sciences, Beijing 100190, China (xiemanting@lsec.cc.ac.cn)}
}
\date{}
\maketitle
\begin{abstract}
In this paper, we propose a computable error estimate of the Gross-Pitaevskii equation for
ground state solution of Bose-Einstein condensates by general conforming finite element methods on general meshes.
Based on the proposed error estimate, asymptotic lower bounds of the smallest eigenvalue and ground state energy can be obtained.
Several numerical examples are presented to validate our theoretical results in this paper.
\end{abstract}

\section{Introduction}
Bose-Einstein condensation (BEC) is one of the most important science discoveries in the last century.
When a dilute gas of trapped bosons (of the same species) is cooled down to ultra-low temperatures
(close to absolute zero), BEC could be formed \cite{DalGioPitaString,LiebSeiYang}.
Since 1995, the first experimental achievement of BECs in dilute ${}^{87}Rb$ gases, the Gross-Pitaevskii equation (GPE) \cite{Gross,JinLevermoreMcLaughlin}
has been used extensively to describe the single particle properties of BECs.
So far, it is found that the results obtained
by solving the GPE have showed excellent agreement with most of the experiments
\cite{AnglinKetterle,Cornell,DalGioPitaString,HauBuschLiuDutton}.

A lot of numerical methods for the computation of the time-independent GPE for ground state and the
time-dependent GPE for finding the dynamics of a BEC have been proposed,
for example: a time-splitting spectral method \cite{BaoJakschMarkowich,BaoJinMarkowich2},
a Crank-Nicolson type finite difference method \cite{Adhikari,AdhikariMuruganandam},
a Runge-Kutta type method \cite{EdwardsBurnett}, an explicit imaginary-time algorithm \cite{Cerimele.et.al.},
 a DIIS (direct inversion in the iterated subspace) method \cite{SchneiderFeder} and
 the optimal damping algorithm \cite{Cances,CancesChakirMaday}, multigrid methods
 and multilevel correction method \cite{XieXie} and so on.

For simplicity, in this paper, we are concerned with the following non-dimensionalized GPE problem:
Find $(\lambda,u)\in \mathbb{R}\times H^1(\Omega)$ such that
\begin{equation}\label{GPEsymply}
\left\{
\begin{array}{rcl}
-\Delta u + Wu + \zeta |u|^2u &=& \lambda u,\ \ \  {\rm in}\  \Omega,\\
u &=& 0,\ \ \ \ \  {\rm on}\ \partial \Omega,\\
\int_\Omega|u|^2d\Omega &=& 1,
\end{array}
\right.
\end{equation}
where $\Omega \subset \mathbb{R}^d$ $(d = 1,2,3)$ denotes the computing
domain which has the cone property \cite{Adams}, $\zeta$ is some positive constant and
$W(x) = \gamma_1 x^2_1 +\ldots + \gamma_d x^2_d \geq 0$
with  $\gamma_1, \ldots, \gamma_d > 0$ \cite{BaoTang,ZhouBEC}.  The ground state energy $E$ of
 BEC can be given by the following equations (cf. \cite{LiebSeiYang}):
\begin{equation}\label{Energy_Eq}
E = \int_\Omega\left(|\nabla u|^2 + W|u|^2 + \frac{\zeta}{2}|u|^4\right){\rm d}\Omega
= \lambda - \int_\Omega\frac{\zeta}{2}|u|^4{\rm d}\Omega,
\end{equation}
where $(\lambda,u)$ is the smallest eigenpair of (\ref{GPEsymply}).

The lower bounds of smallest eigenvalue of (\ref{GPEsymply}) and ground state energy (\ref{Energy_Eq})
are very critical questions that many people care about. So far, there have developed
 some methods to get lower bound of linear symmetric eigenvalue problems, i.e., the nonconforming finite element methods
 (see e.g., \cite{ArmentanoDuran,HuHuangLin,LinLuoXie_lowerbound,LinXie_lowerbound,LinXie_lowerbound2,LinXieLuoLiYang,LuoLinXie,YangZhangLin,ZhangYangChen}), interpolation constant based methods (see e.g., \cite{Liu,LiuOishi}) and computational error estimate methods (see e.g., \cite{CarstensenGallistl,CarstensenGedicke,SebestovaVejchodsky}). But there are no any result about the lower bounds 
 of the nonlinear eigenvalue problems. This paper is the first attempt to produce the lower bounds of the nonlinear eigenvalue problems. 

In this paper, we are concerned with the computable error estimates for the ground state of GPE by the finite element
method. As we know, the priori error estimates can only give the asymptotic convergence order. The a posteriori
error estimates are very important for the mesh adaption process. About the a posteriori error estimate for the
partial differential equations by the finite element method, please refer to
\cite{AinsworthOden,BabuskaRheinboldt_1,BabuskaRheinboldt_2,BrennerScott,NeittaanmakiRepin,Repin,Verfurth}
 and the references cited therein. This paper is propose a computable method to obtain asymptotic
 upper bound of the error estimate for the ground state eigenfunction approximation by
 the general conforming finite element methods on the general meshes.
 The approach is based on complementary energy method from
 \cite{HaslingerHlavacek,NeittaanmakiRepin,Repin,Vejchodsky_1,Vejchodsky_2}.
Based on the asymptotic upper bound of the eigenfunction approximation, we can also give the asymptotic lower bounds
of the smallest eigenvalue and ground state energy, which is another contribution of this paper.

An outline of the paper goes as follows. In Section \ref{Section_FEM}, we introduce the
finite element method for GPE. An asymptotic upper bound for the error estimate of the smallest eigenpair
approximation is given in Section \ref{Section_Upper_Bound}. In Section \ref{Section_Lower_Bound},
an asymptotic lower bounds of the smallest eigenvalue and ground state energy are also obtained based on the results in Section \ref{Section_Upper_Bound}.
 Some numerical examples are presented in Section \ref{Section_Numerical_Examples} to validate our theoretical results.
 Some concluding remarks are given in the last section.

\section{Finite element method for GPE}\label{Section_FEM}

In this section, we introduce some notation and the finite element method
for the GPE (\ref{GPEsymply}).
We shall use the standard notation for the Sobolev spaces $W^{s,p}(\Omega)$ and their
associated norms $\|\cdot\|_{s,p,\Omega}$ and seminorms $|\cdot|_{s,p,\Omega}$
(see, e.g., \cite{Adams}). For $p=2$, we denote
$H^s(\Omega)=W^{s,2}(\Omega)$ and $H_0^1(\Omega)=\{v\in H^1(\Omega):\ v|_{\partial\Omega}=0\}$,
where $v|_{\partial\Omega}=0$ is in the sense of trace,
$\|\cdot\|_{s,\Omega}=\|\cdot\|_{s,2,\Omega}$. In this paper, we set $V=H_0^1(\Omega)$
and use $\|\cdot\|_s$ to denote $\|\cdot\|_{s,\Omega}$ for simplicity.

For the aim of finite element
discretization, we define the corresponding weak form for (\ref{GPEsymply}) as follows: Find $(\lambda,u)\in \mathbb{R}\times V$ such that $b(u,u) = 1$ and
\begin{equation}\label{GPEweakform}
\hat{a}(u,v) = \lambda b(u,v),\ \ \  \forall v \in V,
\end{equation}
where
\begin{eqnarray*}
\hat{a}(u,v) &:=& a(u,v) + \int_\Omega\big((W-1)uv + \zeta|u|^2 uv\big)d\Omega,\\
 a(u,v) &:=& \int_\Omega\big(\nabla u\nabla v + uv\big)d\Omega,
\ \ b(u,v) ~:=~ \int_\Omega uvd\Omega.
\end{eqnarray*}

Obviously, $a(v,v) \geq 0 \mbox{ for all }v\in V$. We define $\|v\|_a = \sqrt{a(v,v)}\ {\rm and}\ \|v\|_b = \sqrt{b(v,v)}\mbox{ for all }v\in V$.
The following Rayleigh quotient expression holds for the smallest eigenvalue $\lambda$
\begin{equation}\label{RQ_u}
\lambda = \frac{\hat{a}(u,u)}{b(u,u)}.
\end{equation}
Now, let us demonstrate the finite element method \cite{BrennerScott,Ciarlet}
for the problem (\ref{GPEweakform}). First we
generate a shape-regular decomposition of the computing domain
$\Omega \subset \mathbb{R}^d$ $(d = 2,3)$ into triangles or rectangles for $d = 2$
(tetrahedrons or hexahedrons for $d = 3$) and the diameter of a cell $K \in \mathcal{T}_h$ is
denoted by $h_K$. The mesh diameter $h$ describes the maximum diameter of all cells
$K \in \mathcal{T}_h$. Based on the mesh $\mathcal{T}_h$, we construct the
conforming finite element space denoted by $V_h\subset V$. We assume that $V_h\subset V$
is a family of finite-dimensional spaces that satisfy the following assumption:
\begin{equation}\label{approximation_fem}
\lim_{h\rightarrow 0}\inf_{v \in V_h} \|w - v\|_a = 0,\ \ \ \forall w\in V.
\end{equation}

The standard finite element method for (\ref{GPEweakform}) is to solve the following
 eigenvalue problem: Find $(\lambda_h,u_h)\in \mathbb{R}\times V_h$ such that
 $b(u_h,u_h) = 1$ and
\begin{equation}\label{GPEfem}
\hat{a}(u_h,v_h) = \lambda_h b(u_h,v_h),\ \ \   \forall v_h \in V_h.
\end{equation}
Then we define
\begin{equation}\label{delta}
\delta_h(u) := \inf_{v_h\in V_h}\|u - v_h\|_a.
\end{equation}
From (\ref{GPEfem}), we know the following Rayleigh quotient for $\lambda_h$ holds
\begin{equation}\label{RQ_uh}
\lambda_h = \frac{\hat{a}(u_h,u_h)}{b(u_h,u_h)}.
\end{equation}
The approximation $E_h$ of ground state energy $E$ for BEC can be given by the following equations:
\begin{equation}\label{Energy_Eq_FEM}
E_h = \int_\Omega\left(|\nabla u_h|^2 + W|u_h|^2 + \frac{\zeta}{2}|u_h|^4\right){\rm d}\Omega
= \lambda_h - \int_\Omega\frac{\zeta}{2}|u_h|^4{\rm d}\Omega,
\end{equation}
where $(\lambda_h,u_h)$ is the smallest eigenpair of (\ref{GPEfem}) that is the approximation for the smallest 
eigenpair of (\ref{GPEsymply}).
\begin{lemma}\label{lemma:Maday}
(\cite[Theorem 1]{CancesChakirMaday})
There exists $h_0 > 0$, such that for all $0 < h < h_0$, the smallest eigenpair approximation
 $(\lambda_h,u_h)$ of (\ref{GPEfem}) having the following error estimates
\begin{eqnarray}
\|u - u_h\|_a &\leq& C_u\delta_h(u),\\
\|u - u_h\|_0 &\leq& C_u\eta_a(h)\|u - u_h\|_a \leq C_u\eta_a(h)\delta_h(u),\\
|\lambda - \lambda_h| &\leq& C_u\|u - u_h\|^2_a
+ C_u\|u - u_h\|_0\leq C_u\eta_a(h)\delta_h(u),
\end{eqnarray}
where $\eta_a(h)$ is defined as follows:
\begin{eqnarray}\label{eta_a_h}
\eta_a(h)=\|u - u_h\|_a+ \sup_{f\in L^2(\Omega),\|f\|_0=1}\inf_{v_h\in V_h}\|Tf-v_h\|_a
\end{eqnarray}
with the operator $T$ being defined as follows: Find $Tf\in u^{\perp}$ such that
\begin{eqnarray*}
a(Tf,v)+2(\zeta |u|^2(Tf),v)- (\lambda(Tf),v)=(f,v),\ \ \ \ \forall v\in u^{\perp},
\end{eqnarray*}
where $u^{\perp}=\big\{v\in H_0^1(\Omega): | \int_{\Omega}uvd\Omega=0\big\}$, 
here and hereafter $C_u$ (with or without subscripts) is some constant depending on eigenpair
$(\lambda,u)$ but independent of the mesh size $h$.
\end{lemma}
\section{Complementarity based a posteriori error estimator}\label{Section_Upper_Bound}
First, we define $H({\rm div};\Omega):= \{\bfp\in (L^2(\Omega))^d:\dive \bfp\in L^2(\Omega)\}\ (d = 2,3)$
and introduce the following Green's theorem.

\begin{lemma}\label{Monk}
Let $\Omega\subset \mathbb{R}^d\ (d=2,3)$ be a bounded Lipschitz domain with unit outward normal $\nu$ to $\partial\Omega$.
Then the following Green's formula holds
\begin{eqnarray}\label{Divergence_Equality}
\int_{\Omega}v{\rm div}\bfp d\Omega+\int_{\Omega}\bfp\cdot\nabla vd\Omega
=\int_{\partial\Omega}v\bfp\cdot\nu ds,\ \ \forall\bfp\in \bf W  \ {\rm and}\  \forall v\in H^1(\Omega), 
\end{eqnarray}
where $\mathbf W:=H({\rm div};\Omega)$.
Especially, we have
\begin{eqnarray}\label{Div_Equality}
\int_{\Omega}v{\rm div}\bfp d\Omega+\int_{\Omega}\bfp\cdot\nabla vd\Omega = 0,\ \ \ \forall \bfp\in \mathbf W\ {\rm and}\  \forall v\in V.
\end{eqnarray}
\end{lemma}

\begin{theorem}
Assume $h<h_0$, the given smallest eigenpair approximation $(\lambda_h,u_h)\in\mathbb{R}\times V_h$  has the
following error estimates:
\begin{equation}\label{Upper_Bound}
\|u-u_h\|_a \leq \min_{\bfp\in\bfW}\frac{1}{1-\alpha\eta_a(h)}\eta(\lambda_h,u_h,\bfp),
\end{equation}
where $\eta(\lambda_h,u_h,\bfp)$ is defined as follows
\begin{eqnarray}\label{Definition_Eta}
\eta(\lambda_h,u_h,\bfp) = \big(\|\lambda_hu_h - Wu_h - \zeta|u_h|^2 u_h + \dive\bfp\|_0^2 + \|\bfp - \nabla u_h\|_0^2\big)^{1/2},
\end{eqnarray}
$\alpha = C_u\|u\|_a + C_u\|W -1 - \lambda_h\|_{L^{\infty}(\Omega)} + 2C_u|\zeta|(\|u\|_a^2 + \|u_h\|_a^2)$,
and the constant $C$ is independent of the mesh size $h$, vector function $\bfp$ and eigenfunction approximation $u_h$.
\end{theorem}
\begin{proof}
From (\ref{GPEweakform}), (\ref{GPEfem}) and (\ref{Div_Equality}), for all $\bfp\in\bfW$ and $v\in V$,
we have
\begin{eqnarray*}
a(u-u_h,v) &=& a(u,v) - a(u_h,v)\\
&=& \lambda b(u,v) - \int_{\Omega}\big((W-1)uv + \zeta|u|^2 uv\big)d\Omega \\
& & \quad -~ \int_\Omega(\nabla u_h\cdot\nabla v + u_hv)d\Omega + \int_\Omega\dive\bfp vd\Omega + \int_\Omega\bfp\cdot\nabla vd\Omega\\
&=& (\lambda_hu_h - Wu_h - \zeta|u_h|^2 u_h + \dive\bfp, v) + (\bfp - \nabla u_h, \nabla v)\\
& & \quad +~ \big(\lambda u - \lambda_hu_h + (W-1)(u_h - u) + \zeta|u_h|^2_h - \zeta|u|^2 u, v\big)\\
&\leq& \|\lambda_hu_h - Wu_h - \zeta|u_h|^2 u_h + \dive\bfp\|_0\|v\|_0 + \|\bfp - \nabla u_h\|_0\|\nabla v\|_0\\
& & + \big((\lambda - \lambda_h) u - (W -1 - \lambda_h)(u - u_h) + \zeta(|u|^2 u - |u_h|^2 u_h), v\big)\\
&\leq& \big(\|\lambda_hu_h - Wu_h - \zeta|u_h|^2 u_h + \dive\bfp\|_0^2 + \|\bfp - \nabla u_h\|_0\big)^{1/2}\|v\|_a\\
& & \quad +~ \big(|\lambda - \lambda_h|\|u\|_a + \|W -1 - \lambda_h\|_{L^{\infty}(\Omega)}\|u - u_h\|_0 \\
& & \quad +~ 2|\zeta|(\|u\|_{L^6(\Omega)}^2 + \|u_h\|_{L^6(\Omega)}^2)\|u-u_h\|_0\big)\|v\|_{L^6(\Omega)}\\
&\leq& \big(\|\lambda_hu_h - Wu_h - \zeta|u_h|^2 u_h + \dive\bfp\|_0^2 + \|\bfp - \nabla u_h\|_0\big)^{1/2}\|v\|_a\\
& & \quad +~ \big(|\lambda - \lambda_h|\|u\|_a + \|W -1 - \lambda_h\|_{L^{\infty}(\Omega)}\|u - u_h\|_0 \\
& & \quad +~ 2C|\zeta|(\|u\|_a^2 + \|u_h\|_a^2)\|u-u_h\|_0\big)\|v\|_a.
\end{eqnarray*}
Together with Lemma \ref{lemma:Maday}, we have
\begin{eqnarray*}
a(u-u_h,v) &\leq& \big(\|\lambda_hu_h - Wu_h - \zeta|u_h|^2 u_h + \dive\bfp\|_0^2 + \|\bfp - \nabla u_h\|_0\big)^{1/2}\|v\|_a\\
 & & \quad +~ \alpha\eta_a(h)\|u - u_h\|_0\|v\|_a,
\end{eqnarray*}
where $\alpha = C_u\|u\|_a + C_u\|W -1 - \lambda_h\|_{L^{\infty}(\Omega)} + 2C_uC|\zeta|(\|u\|_a^2 + \|u_h\|_a^2)$.

That means we can draw the conclusion that
$$\|u-u_h\|_a \leq \eta(\lambda_h,u_h,\bfp) + \alpha\eta_a(h)\|u-u_h\|_a, \ \ \forall \bfp\in\bfW.$$
Then the desired result (\ref{Upper_Bound}) can be obtained by the arbitrariness of $\bfp\in\bfW$ and the proof is complete.
\end{proof}

Now, we introduce how to choose $\bfp^*\in \bf W$ such that
\begin{equation}\label{Optimization}
\eta(\lambda_h,u_h,\bfp^*) = \min_{\bfp\in\bfW}\eta(\lambda_h,u_h,\bfp).
\end{equation}

\begin{lemma}(\cite{Vejchodsky_1})\label{Dual_lemma}
The optimization problem (\ref{Optimization}) is equivalent to the following partial differential equation:
Find $\bfp^*\in\bfW$ such that
\begin{eqnarray}\label{Dual_Problem}
a^*(\bfp^*,\bfq)&=& \mathcal F^*(\bfq),\ \ \ \ \forall \bfq\in \mathbf W,\
\end{eqnarray}
where
\begin{eqnarray*}
a^*(\bfp^*,\bfp)&=&\int_{\Omega}\big({\rm div}\bfp^*{\rm div}\bfq
+\bfp^*\cdot\bfq\big)d\Omega, \\
\mathcal F^*(\bfq)&=&-\int_{\Omega}\big(\lambda_h u_h - \zeta |u_h|^2u_h - (W-1)u_h,\dive\bfq\big)d\Omega.
\end{eqnarray*}
Moreover, $a^*(\cdot,\cdot)$ defines an inner product for the space $\mathbf W$. The corresponding norm
is $|||\bfp|||_{*}=a^*(\bfp,\bfp)$, and the dual problem (\ref{Dual_Problem}) has a unique solution.
\end{lemma}

Now, we states some properties for the estimator $\eta(\lambda_h,u_h,\bfp)$.
\begin{lemma}(\cite{Vejchodsky_1})\label{Optimization_Property_Lemma}
Assume $\bfp^*$ be the solution of the dual problem (\ref{Dual_Problem}) and let $\lambda_h\in \mathcal{R}$,
$u_h\in V$ and $\bfp\in\bfW$ be arbitrary. Then the following equality holds
\begin{eqnarray}\label{Optimization_Property}
\eta^2(\lambda_h,u_h,\bfp)&=&\eta^2(\lambda_h,u_h,\bfp^*)+|||\bfp^*-\bfp|||_*^2.
\end{eqnarray}
\end{lemma}

Choosing a certain approximate solution $\bfp_h\in \mathbf W$ of the dual problem (\ref{Dual_Problem}), we can give a
computable asymptotic upper bound of the error estimate for the eigenfunction approximation.

\begin{corollary}\label{Upper_Bound_Computable_corollary}
Assume $(\lambda,u)\in\mathbb{R}\times V$ is the ground state solution of (\ref{GPEweakform}),
and $(\lambda_h,u_h)$ is the solution of eigenvalue problem (\ref{GPEfem}).
Then the error estimate has the following upper bound
\begin{eqnarray}\label{Upper_Bound_Computable}
\|u-u_h\|_a&\leq&\frac{1}{1-\alpha\eta_a(h)}\eta(\lambda_h,u_h,\bfp_h).
\end{eqnarray}
\end{corollary}

Hereafter, we also discuss the efficiency of the estimator $\eta(\lambda_h,u_h,\mathbf y^*)$ and
$\eta(\lambda_h,u_h,\mathbf y_h)$.
\begin{theorem}\label{Efficiency_Theorem}
Let $(\lambda,u)$ be the smallest exact solution of (\ref{GPEweakform}), $(\lambda_h,u_h)$ be
the smallest eigenpair approximation of (\ref{GPEfem}),
and $\bfp^*$ be the solution of (\ref{Optimization}).
Then we have
\begin{eqnarray}\label{Efficiency}
\theta_1\|u-u_h\|_a \leq
\eta(\lambda_h,u_h,\mathbf \bfp^*)\leq \theta_2\|u-u_h\|_a,
\end{eqnarray}
where
\begin{eqnarray*}
\theta_1:={1-\alpha\eta_a(h)}\ \ \ {\rm and}\ \ \
\theta_2:=\sqrt{1+C_1\eta_a(h)}.
\end{eqnarray*}
Further, we have the following asymptotic property of the estimator
\begin{eqnarray}\label{Exactness}
\lim_{h\rightarrow 0}\frac{\eta(\lambda_h,u_h,\bfp^*)}{\|u-u_h\|_a}=1.
\end{eqnarray}
\end{theorem}
\begin{proof}
The left-hand side inequality of (\ref{Efficiency}) is a direct conclusion of (\ref{Upper_Bound}). Next, we prove the
right-hand side one of (\ref{Efficiency}).

From the definition (\ref{Definition_Eta}), the GPE (\ref{GPEweakform})
and $\nabla u\in\mathbf W$, we have
\begin{eqnarray}\label{Inequality_4}
\eta^2(\lambda_h,u_h,\nabla u) &=& \|\lambda_hu_h - Wu_h - \zeta|u_h|^2 u_h- (\lambda u - Wu - \zeta|u |^2 u)\|_0^2\nonumber\\
& & \quad\quad ~+ \|\nabla u - \nabla u_h\|_0^2.
\end{eqnarray}
Then combining (\ref{Optimization_Property}), (\ref{Inequality_4}) and Lemma \ref{lemma:Maday},
 the following estimates hold
\begin{eqnarray}\label{Inequality_3}
&&\eta^2(\lambda_h,u_h,\bfp^*)\leq \eta^2(\lambda_h,u_h,\nabla u)\nonumber\\
&=&\|\lambda_hu_h - Wu_h - \zeta|u_h|^2 u_h- (\lambda u - Wu - \zeta|u |^2 u)\|_0^2\nonumber\\
& &\ \ ~ + \|\nabla u - \nabla u_h\|_0^2\nonumber\\
&=&\|\lambda_hu_h - Wu_h - \zeta|u_h|^2 u_h- (\lambda u - Wu - \zeta|u |^2 u)\|_0^2\nonumber\\
& &\ \ ~ + \|u-u_h\|_a^2 - \|u-u_h\|_0^2\nonumber\\
&=&2\|(\lambda_h-\lambda)u_h - (\lambda-W)(u_h-u)\|_0^2 + 2\|\zeta|u_h|^2u_h - \zeta|u|^2u\|_0^2\nonumber\\
& &\ \ ~ + \|u-u_h\|_a^2 - \|u-u_h\|_0^2\nonumber\\
&\leq&4|\lambda_h-\lambda|^2\|u_h\|_0^2 + 4\|\lambda-W\|_{L^{\infty}}^2\|u_h-u\|_0^2
 + 2\|\zeta(|u_h|^2 + uu_h +|u|^2)(u_h-u)\|_0^2\nonumber\\
& &\ \ ~ + \|u-u_h\|_a^2 - \|u-u_h\|_0^2\nonumber\\
&\leq&4|\lambda_h-\lambda|^2\|u_h\|_0^2 + 4\|\lambda-W\|_{L^{\infty}}^2\|u_h-u\|_0^2 \nonumber\\
&& + 2|\zeta|^2\left(\int_\Omega(|u_h|^2 + uu_h +|u|^2)^6{\rm d}\Omega\right)^{1/3}\left(\int_\Omega(u_h-u)^6{\rm d}\Omega\right)^{1/6}\left(\int_\Omega(u_h-u)^2{\rm d}\Omega\right)^{1/2}\nonumber\\
& &\ \ ~ + \|u-u_h\|_a^2 - \|u-u_h\|_0^2\nonumber\\
&\leq&4|\lambda_h-\lambda|^2\|u_h\|_0^2 + 4\|\lambda-W\|_{L^{\infty}}^2\|u_h-u\|_0^2 \nonumber\\
& &\ \ ~ + 4|\zeta|^2\left(\|u_h\|_{L^{12}(\Omega)} + \|u\|_{L^{12}(\Omega)}\right)\|u_h - u\|_{L^6(\Omega)}\|u_h-u\|_0\nonumber\\
& &\ \ ~ + \|u-u_h\|_a^2 - \|u-u_h\|_0^2\nonumber\\
&\leq&4|\lambda_h-\lambda|^2\|u_h\|_0^2 + 4\|\lambda-W\|_{L^{\infty}}^2\|u_h-u\|_0^2 \nonumber\\
& &\ \ ~ + 4C|\zeta|^2\left(\|u_h\|_a + \|u\|_a\right)\|u_h - u\|_a\|u_h-u\|_0\nonumber\\
& &\ \ ~ + \|u-u_h\|_a^2 - \|u-u_h\|_0^2\nonumber\\
&\leq&\big(1+C_1\eta_a(h)\big)\|u-u_h\|_a^2,
\end{eqnarray}
where $C_1 = 4(C_u^2\eta_a(h)\|u_h\|_0^2+C_u^2(\|\lambda-W\|_{L^{\infty}}^2-1)\eta_a(h)
 + C_uC|\zeta|^2(\|u_h\|_a + \|u\|_a)) $.

The inequality (\ref{Inequality_3}) leads to the right-hand side inequality of (\ref{Efficiency}).
Hence, the desired result (\ref{Exactness}) can be deduced easily from the
fact that $\delta_h(u)\rightarrow 0$ and $\eta_a(h)\rightarrow 0$ as $h\rightarrow 0$.
\end{proof}

\begin{corollary}
Assume the conditions of Theorem \ref{Efficiency_Theorem} holds and there exist a constant such that
$|||\bfp^*-\bfp_h\||_*\leq \widehat{C}\|u-u_h\|_a$. Then the following efficiency holds
\begin{eqnarray}\label{Efficiency_2}
\eta(\lambda_h,u_h,\bfp_h)&\leq&C_2\|u-u_h\|_a,
\end{eqnarray}
where $C_2$ is a constant defined by $C_2:=\sqrt{\theta_2^2+\widehat{C}^2}$. 
\end{corollary}

\begin{proof}
First from (\ref{Optimization_Property}) and (\ref{Efficiency}), we have
\begin{eqnarray}
\eta^2(\lambda_h,u_h,\bfp_h)&=&\eta^2(\lambda_h,u_h,\bfp^*)+|||\bfp^*-\bfp_h|||_*^2\nonumber\\
&\leq&\theta_2^2\|u-u_h\|_a^2+\widehat{C}^2\|u-u_h\|_a^2\nonumber\\
&\leq&(\theta_2^2+\widehat{C}^2)\|u-u_h\|_a^2.
\end{eqnarray}
Then the desired result (\ref{Efficiency_2}) can be obtained 
and the proof is complete.
\end{proof}

\section{Asymptotic lower bound of the first eigenvalue}\label{Section_Lower_Bound}
In this section, based on the upper bound for the error estimate of the first eigenfunction approximation,
we give an asymptotic lower bound of the smallest eigenvalue. Actually, the process is direct since we have
the following Rayleigh quotient expansion. 
\begin{lemma}\label{RQ}
Assume $(\lambda,u)\in\mathbb{R}\times V$ is the eigenpair of the original problem (\ref{GPEweakform}),
$(\lambda_h,u_h)\in\mathbb{R}\times V_h$ is the
eigenpair of the discrete problem (\ref{GPEfem}). We have the following expansion:
\begin{eqnarray}\label{RQexpansion}
\lambda_h - \lambda &=& \frac{a(u_h-u,u_h-u) - \lambda b(u_h-u,u_h-u)
+ \int_\Omega\big((W-1)(u_h-u)^2d\Omega}{b(u_h,u_h)}\nonumber\\
 & & \ \ ~+\frac{\int_{\Omega}\zeta(|u|^2u - |u|^2u_h - |u_h|^2u - |u_h|^2u_h)(u_h-u)d\Omega}{b(u_h,u_h)}.
\end{eqnarray}
\end{lemma}
\begin{proof}
From (\ref{GPEweakform}), (\ref{RQ_u}), (\ref{GPEfem}), (\ref{RQ_uh}), and direct computations, we have
\begin{eqnarray*}
& & \lambda_h - \lambda = \frac{\hat{a}(u_h,u_h) - \lambda b(u_h,u_h)}{b(u_h,u_h)}\\
&=& \frac{a(u_h,u_h) + \int_\Omega\big((W-1)u_hu_h + \zeta|u_h|^2 u_hu_h\big)d\Omega - \lambda b(u_h,u_h)}{b(u_h,u_h)}\\
&=& \frac{a(u_h-u,u_h-u) + 2a(u,u_h) - a(u,u) }{b(u_h,u_h)}\\
& & \ \ +~ \frac{\int_\Omega\big((W-1)u_hu_h + \zeta|u_h|^2 u_hu_h\big)d\Omega - \lambda b(u_h,u_h)}{b(u_h,u_h)}\\
&=& \frac{a(u_h-u,u_h-u) + 2\lambda b(u,u_h) - 2\int_{\Omega}\big((W-1)uu_h + \zeta|u|^2uu_h\big)d\Omega}{b(u_h,u_h)}\\
& & \ \ +~ \frac{-\lambda b(u,u) + \int_{\Omega}\big((W-1)uu + \zeta|u|^2uu\big)d\Omega}{b(u_h,u_h)}\\
& & \ \ +~ \frac{\int_\Omega\big((W-1)u_hu_h + \zeta|u_h|^2 u_hu_h\big)d\Omega - \lambda b(u_h,u_h)}{b(u_h,u_h)}\\
&=& \frac{a(u_h-u,u_h-u) - \lambda b(u_h-u,u_h-u) + \int_\Omega\big((W-1)(u - u_h)^2d\Omega }{b(u_h,u_h)}\\
& & \ \ ~+\frac{ \int_{\Omega}\zeta(|u|^2u_h + |u_h|^2u + |u_h|^2u_h - |u|^2u)(u-u_h)d\Omega}{b(u_h,u_h)}.
\end{eqnarray*}
This is the desired result (\ref{RQexpansion}) and the proof is complete.
\end{proof}

\begin{theorem}\label{Lower_bound_Lambda}
Assume the conditions of Lemma \ref{RQ} and the normalization condition $b(u_h,u_h) = 1$ hold.
Then we have the following error estimate:
\begin{eqnarray}\label{RQ_inequality}
\lambda_h - \lambda \leq \frac{C_u\delta_h(u) + C_3\eta_a(h)}{1-\alpha\eta_a(h)}\eta(\lambda_h,u_h,\bfp_h),
\end{eqnarray}
where $C_3 = \big(\|W-1\|_{L^{\infty}(\Omega)}\|u_h-u\|_0+ C|\zeta|(\|u\|^3_a + \|u_h\|^2_a\|u\|_a+\|u_h\|^3_a+ \|u_h\|^2_a\|u_h\|_a)\big)C_u$.

Moreover, if $h$ is small enough such that $\frac{C_u\delta_h(u) + C_3\eta_a(h)}{1-\alpha\eta_a(h)}\leq 1$,
the following explicit and asymptotic result holds
\begin{equation}\label{Asymptotic_lower_Bound}
\lambda_h^L:=\lambda_h-\eta(\lambda_h,u_h,\bfp_h) \leq \lambda,
\end{equation}
where $\lambda_h^L$ denotes an asymptotic lower bound of the first eigenvalue $\lambda$.
\end{theorem}
\begin{proof}
From (\ref{RQexpansion}) and $b(u_h,u_h) = 1$, we have the following estimates
\begin{eqnarray}\label{RQE_estimate1}
\lambda_h - \lambda &\leq& \|u_h-u\|_a^2 - \lambda \|u_h-u\|_0^2 + \|W-1\|_{L^{\infty}(\Omega)}\|u_h-u\|_0^2\nonumber\\
& & \ \ ~+\|\zeta(|u|^2 u - |u_h|^2u - |u_h|^2u_h - |u|^2u_h)\|_0\|u_h - u\|_0\nonumber\\
&\leq& \|u_h-u\|_a^2 - \lambda \|u_h-u\|_0^2 + \|W-1\|_{L^{\infty}(\Omega)}\|u_h-u\|_0^2\nonumber\\
& & \ \ ~+|\zeta|\big(\|u\|^3_{L^6(\Omega)} + \|u_h\|^2_{L^6(\Omega)}\|u\|_{L^6(\Omega)}+\|u_h\|^3_{L^6(\Omega)}\nonumber\\
& & \ \ \ ~ + \|u_h\|^2_{L^6(\Omega)}\|u_h\|_{L^6(\Omega)}\big)\|u_h - u\|_0\nonumber\\
&\leq& \|u_h-u\|_a^2 - \lambda \|u_h-u\|_0^2 + \|W-1\|_{L^{\infty}(\Omega)}\|u_h-u\|_0^2\nonumber\\
& & \ \ ~+C|\zeta|\big(\|u\|^3_a + \|u_h\|^2_a\|u\|_a+\|u_h\|^3_a\nonumber\\
& & \ \ \ ~ + \|u_h\|^2_a\|u_h\|_a\big)\|u_h - u\|_0.
\end{eqnarray}
Then using Lemma \ref{lemma:Maday}, (\ref{Upper_Bound_Computable}) and (\ref{RQE_estimate1}), we have
\begin{eqnarray}
\lambda_h - \lambda &\leq& \|u_h-u\|_a^2 + \big(\|W-1\|_{L^{\infty}(\Omega)}\|u_h-u\|_0\nonumber\\
& & \ \ ~+ C|\zeta|\big(\|u\|^3_a + \|u_h\|^2_a\|u\|_a+\|u_h\|^3_a+ \|u_h\|^2_a\|u_h\|_a\big)\|u_h - u\|_0\nonumber\\
&\leq& \big(C_u\delta_h(u) + C_3\eta_a(h)\big)\|u_h - u\|_a\nonumber\\
&\leq&\frac{C_u\delta_h(u) + C_3\eta_a(h)}{1-\alpha\eta_a(h)}\eta(\lambda_h,u_h,\bfp_*)\nonumber\\
&\leq& \frac{C_u\delta_h(u) + C_3\eta_a(h)}{1-\alpha\eta_a(h)}\eta(\lambda_h,u_h,\bfp_h).
\end{eqnarray}
This is the desired result (\ref{RQ_inequality}) and the result (\ref{Asymptotic_lower_Bound}) can
be derived easily.
\end{proof}

\begin{corollary}
Assume the conditions of Theorem \ref{Lower_bound_Lambda} holds. Then we have the following error estimate:
\begin{eqnarray}\label{Energy_inequality}
E_h - E \leq \frac{C_u\delta_h(u) + C_4\eta_a(h)}{1-\alpha\eta_a(h)}\eta(\lambda_h,u_h,\bfp_h),
\end{eqnarray}
where $C_4 = C_3+\frac{C\zeta}{2(1-\alpha\eta_a(h))}(\|u\|_a+\|u_h\|_a)^3$.

Moreover, if $h$ is small enough such that $\frac{C_u\delta_h(u) + C_4\eta_a(h)}{1-\alpha\eta_a(h)}\leq 1$,
the following explicit and asymptotic result holds
\begin{equation}\label{Asymptotic_lower_Bound_Energy}
E_h^L:=E_h-\eta(\lambda_h,u_h,\bfp_h) \leq E,
\end{equation}
where $E_h^L$ denotes an asymptotic lower bound of the ground state energy $E$.
\end{corollary}
\begin{proof}
From (\ref{Energy_Eq}), (\ref{Energy_Eq_FEM}) and (\ref{RQ_inequality}), we have
\begin{eqnarray}\label{Energy_estimate1}
E_h - E &\leq& \lambda_h - \lambda + \int_{\Omega}\frac{\zeta}{2}(|u|^4 - |u_h|^4)d\Omega\nonumber\\
&\leq& \frac{C_u\delta_h(u) + C_3\eta_a(h)}{1-\alpha\eta_a(h)}\eta(\lambda_h,u_h,\bfp_h)\nonumber\\
& & \ \ ~+ \int_{\Omega}\frac{\zeta}{2}(u-u_h)(u+u_h)(u^2+u_h^2)d\Omega\nonumber\\
&\leq& \frac{C_u\delta_h(u) + C_3\eta_a(h)}{1-\alpha\eta_a(h)}\eta(\lambda_h,u_h,\bfp_h)\nonumber\\
&&+ \frac{\zeta}{2}\left(\int_{\Omega}(u-u_h)^2d\Omega\right)^{1/2}\left(\int_{\Omega}(u+u_h)^6d\Omega\right)^{1/6}
\left(\int_{\Omega}(u^2+u_h^2)^{3}d\Omega\right)^{1/3}\nonumber\\
&\leq& \frac{C_u\delta_h(u) + C_3\eta_a(h)}{1-\alpha\eta_a(h)}\eta(\lambda_h,u_h,\bfp_h)\nonumber\\
& & \ \ ~+ \frac{\zeta}{2}\|u-u_h\|_0(\|u\|_{L^6(\Omega)}+\|u_h\|_{L^6(\Omega)})(\|u^2\|_{L^3(\Omega)}+\|u_h^2\|_{L^3(\Omega)})\nonumber\\
&\leq& \frac{C_u\delta_h(u) + C_3\eta_a(h)}{1-\alpha\eta_a(h)}\eta(\lambda_h,u_h,\bfp_h)\nonumber\\
& & \ \ ~+ \frac{C\zeta}{2}\|u-u_h\|_0(\|u\|_a+\|u_h\|_a)^3.
\end{eqnarray}
Then using Lemma \ref{lemma:Maday}, (\ref{Upper_Bound_Computable}) and (\ref{Energy_estimate1}), we have
\begin{eqnarray}\label{Energy_estimate2}
E_h - E &\leq& \frac{C_u\delta_h(u) + C_3\eta_a(h)}{1-\alpha\eta_a(h)}\eta(\lambda_h,u_h,\bfp_h)\nonumber\\
& & \ \ ~+ \frac{C\zeta}{2}\|u-u_h\|_0(\|u\|_a+\|u_h\|_a)^3\nonumber\\
&\leq& \frac{C_u\delta_h(u) + C_4\eta_a(h)}{1-\alpha\eta_a(h)}\eta(\lambda_h,u_h,\bfp_h). 
\end{eqnarray}
This is the desired result (\ref{Energy_inequality}) and the result (\ref{Asymptotic_lower_Bound_Energy}) can
be derived easily.
\end{proof}
\section{Numerical examples}\label{Section_Numerical_Examples}
In this section, two numerical examples are presented to validate the efficiency of
the a posteriori estimate, the upper bound of the error estimate and lower bound of
the first eigenvalue proposed in this paper.

In order to give the a posteriori error estimate $\eta(\lambda_h, u_h, \mathbf p_h)$, we need to solve the
dual problem (\ref{Dual_Problem}). Here, the dual problem (\ref{Dual_Problem}) is solved using the same mesh $\mathcal{T}_h$
and the $H({\rm div};\Omega)$ conforming finite element space $\bfW_h$ is defined as follows \cite{BrezziFortin}
$$\bfW_h^p = \{\bfw\in\bfW:\bfw|_K\in{\rm RT}_p,\forall K\in \mathcal{T}_h\},$$
where ${\rm RT}_p = (\mathcal{P}_p)^d+\mathbf{x}\mathcal{P}_p$. Then the approximate solution $\bfp_h^p\in\bfW_h^p$ of the dual
problem (\ref{Dual_Problem}) is defined as follows: Find $\bfp_h^*\in\bfW_h^p$ such that
\begin{equation}\label{Dual_Problem_fem}
a^*(\bfp_h^*,\bfq_h) = \mathcal{F}(\bfq_h), \ \ \forall \bfq_h\in\bfw_h^p.
\end{equation}
After obtaining $\bfp_h^*$, we can compute the a posteriori error estimate $\eta(\lambda_h,u_h,\bfp_h^*)$ as
in (\ref{Definition_Eta}). Based on $\lambda_h$ and $\eta(\lambda_h,u_h,\bfp_h^*)$, we can obtain the asymptotic
 lower bound of the first eigenvalue $\lambda$ as follows
$$\lambda_h^L:=\lambda_h-\eta(\lambda_h,u_h,\bfp_h^*).$$

Futhermore, we can get an asymptotic lower bound of the ground state energy $E_h^L$ based on $E_h$ and $\eta(\lambda_h,u_h,\bfp_h^*)$:
$$E_h^L := E_h-\eta(\lambda_h,u_h,\bfp_h^*).$$

\begin{example}
In this example, we consider the ground state solution of GPE (\ref{GPEsymply})
for BEC with $\zeta = 1$, $W(x) = x_1^2+x_2^2$ and unit square $\Omega = (0,1)\times(0,1)$.
\end{example}

\begin{figure}[ht]
\centering
\includegraphics[width=6cm,height=5.5cm]{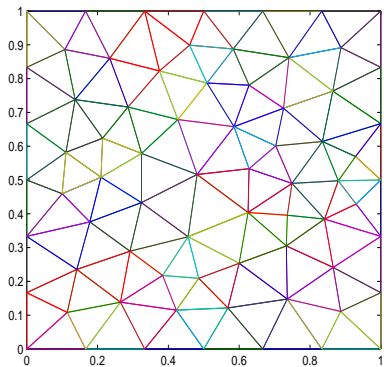}
\caption{\small\texttt The initial mesh for the unit square.}
\label{unit_square_init_mehs}
\end{figure}

In this example, the initial mesh $\mathcal{T}_{h_1}$ is showed in Figure \ref{unit_square_init_mehs} which
is generated by Delaunay method and the mesh size $h_1 = 1/6$. Then we produce a sequence of meshes
$\{\mathcal{T}_{h_i}\}_{i=2}^{6}$ which are obtained by the regular
refinement (connecting the midpoints of each edge) and whose mesh sizes are
$h_2 = 1/12,\ \cdots,\ h_6 = 1/192$. Based on this sequence of meshes, a sequence of
 linear conforming finite element space $\{V_{h_i}\}_{i=1}^{6}$
and $H(\dive;\Omega)$ conforming finite element space $\{\bfW_{h_i}^1\}_{i=1}^{6}$ are built.

First we solve the GPE problem (\ref{GPEweakform}) in $\{V_{h_i}\}_{i=1}^{6}$ and
the dual problem (\ref{Dual_Problem_fem}) in $\{\bfW_{h_i}^1\}_{i=1}^{6}$, respectively.
The corresponding numerical results are presented in Figure \ref{unit_square_eigenpair} which
shows that the a posteriori error estimate $\eta(\lambda_h,u_h,\bfp_h^*)$ is efficient
when we solve the dual problem in $\bfW_h^1$. In Figure \ref{unit_square_eigenpair},
we can find that the eigenvalue approximation $\lambda_h^L$ and ground state energy
approximation $E_h^L$ are really asymptotic lower bounds for the first eigenvalue $\lambda$
and ground state energy $E$, respectively.

\begin{figure}[ht]
\centering
\includegraphics[width=6cm,height=5.5cm]{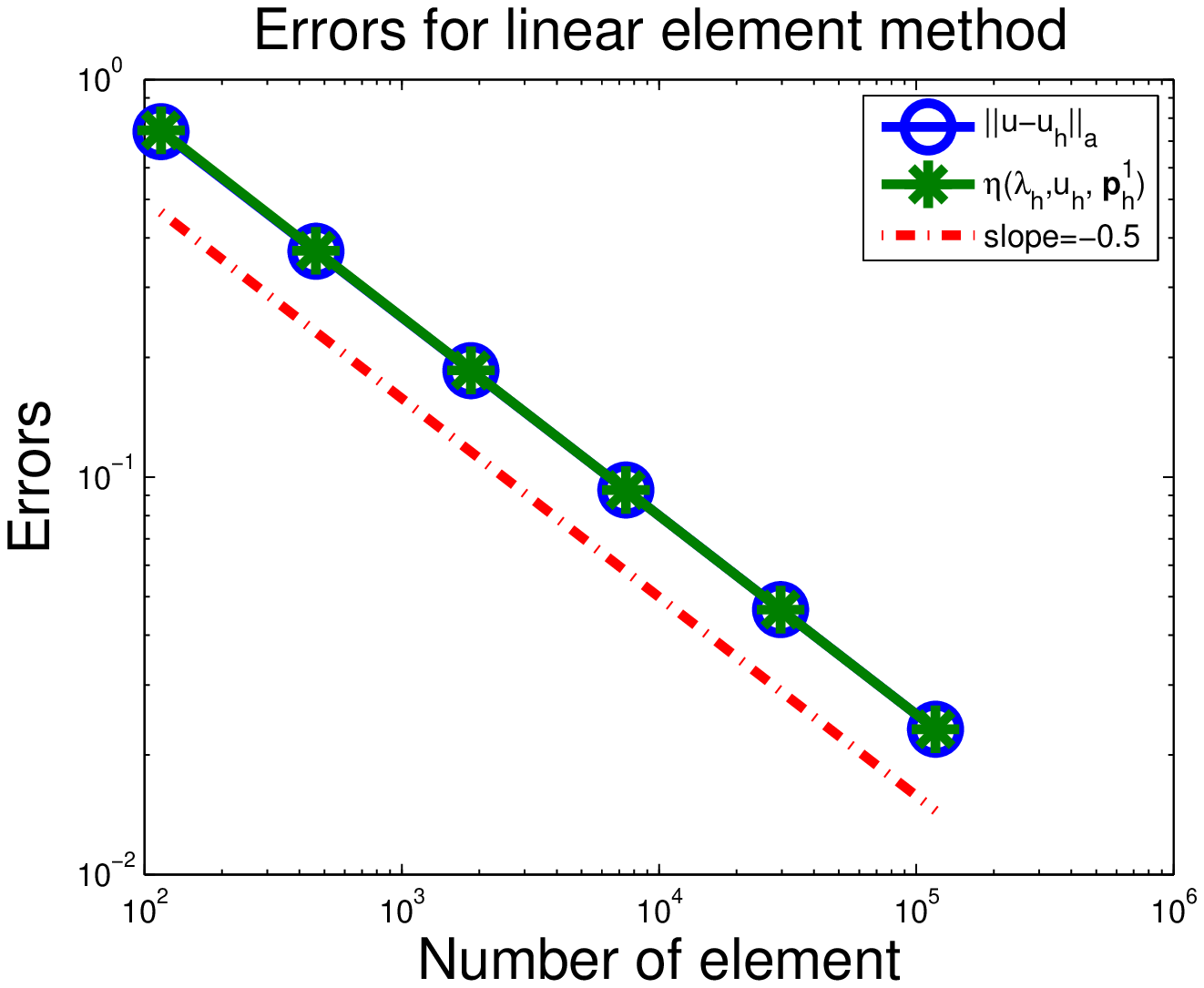}
\includegraphics[width=6cm,height=5.5cm]{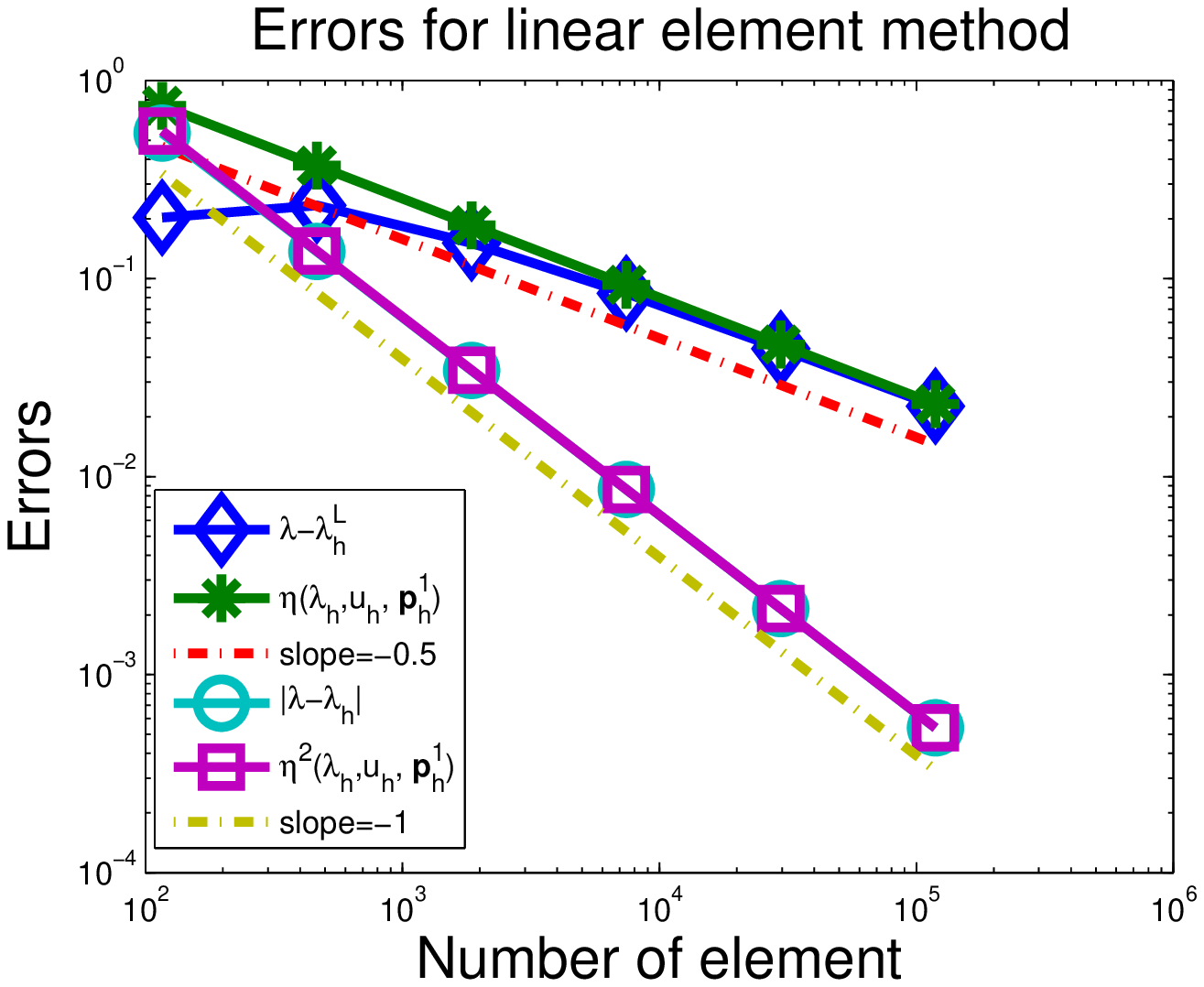}
\includegraphics[width=6cm,height=5.5cm]{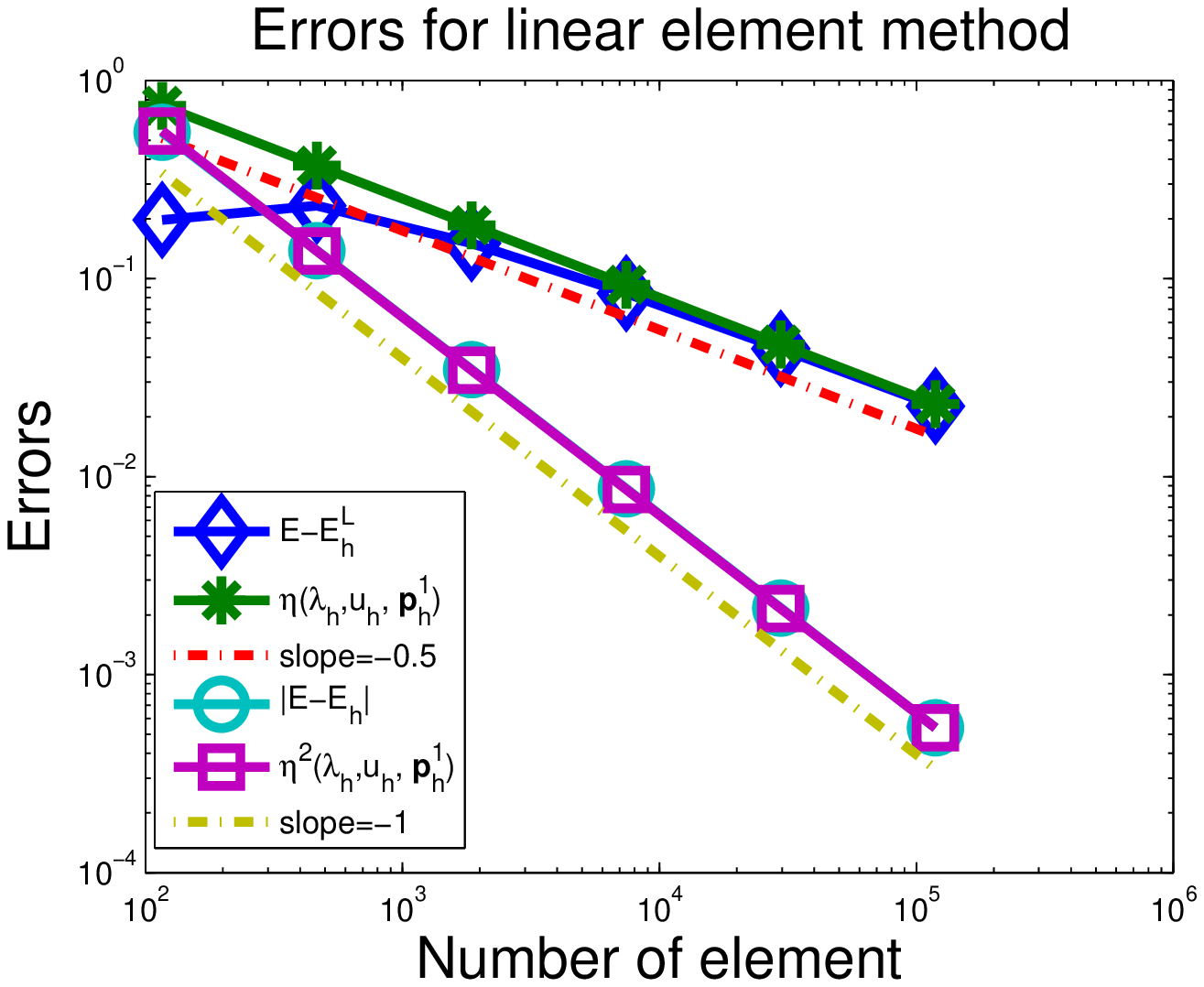}
\caption{\small\texttt The errors for the unit square domain when the eigenvalue problem is solved
by the linear finite element method, where $\eta(\lambda_h,u_h,\bfp_h^1)$ denotes the a
posteriori error estimates $\eta(\lambda_h,u_h,\bfp_h^*)$ when the dual problem is solved in $\bfW_h^1$,
and $\lambda_h^L$ denotes the asymptotic lower bound of the first eigenvalue
 and $E_h^L$ denotes the asymptotic lower bound of the ground state energy.}
\label{unit_square_eigenpair}
\end{figure}

\begin{example}
In this example, we solve the ground state solution of GPE (\ref{GPEsymply})
for BEC with $\zeta = 1$, $W(x) = x_1^2+x_2^2$ on the L shape domain
$\Omega = (-1, 1)\times(-1, 1)/[0, 1)\times(-1, 0]$.
\end{example}

Since $\Omega$ has a re-entrant corner, the singularity
of the first eigenfunction is expected. The convergence order for the eigenvalue
approximation is less than $2$ by the linear finite element method which is the order
predicted by the theory for regular eigenfunctions. Since the exact eigenvalue is not known, we choose
an adequately accurate approximation obtained by the higher order finite element and finer mesh
as the exact first eigenpair for the numerical tests. In order
to treat the singularity of the eigenfunction, we solve the GPE (\ref{GPEweakform})
by the adaptive finite element method (cf. \cite{BrennerScott,ChenHeZhou}).

We present this example to validate the results in this paper also hold on the
adaptive meshes. In order to use the adaptive finite element method, we define the
a posteriori error estimator as follows: Define the element residual $\mathcal{R}_K(\lambda_h,u_h)$ and
the jump residual $\mathcal{J}_E(u_h)$ as follows (see e.g., \cite{ChenHeZhou}):
\begin{eqnarray*}
\mathcal{R}_K(\lambda_h,u_h) &:=& \lambda_hu_h - \zeta|u_h|^2u_h - Wu_h + \Delta u_h\ \ \ {\rm in}\ K\in\mathcal{T}_h,\\
\mathcal{J}_E(u_h) &:=& -\nabla u_h^+\cdot\mathbf{\nu}^+ - \nabla u_h^-\cdot\mathbf{\nu}^- := [[\nabla u_h]]_E\cdot\mathbf{\nu}_E,\ \ {\rm on}\ E\in\mathcal{E}_h,
\end{eqnarray*}
where $\mathcal{E}_h$ denotes the interior edge set in the mesh $\mathcal{T}_h$, $E$ is the common side of elements $K^+$ and $K^-$ with unit outward
normals $\mathbf{\nu}^+$ and $\mathbf{\nu}^-$, respectively, and $\mathbf{\nu}_E = \mathbf{\nu}^-$.

For $K\in\mathcal{T}_h$, we define the local error indicator $\eta_h(\lambda_h,u_h,K)$ by
\begin{equation}
\eta_h^2(\lambda_h,u_h,K):= h_K^2\|\mathcal{R}_K(\lambda_h,u_h)\|_{0,K}^2 + \sum_{E\in\mathcal{E}_h,E\subset\partial K}h_E\|\mathcal{J}_E(u_h)\|_{0,E}^2.
\end{equation}

Then we define the global a posteriori error estimator $\eta_{ad}(\lambda_h, u_h)$ by
\begin{equation}
\eta_{ad}(\lambda_h,u_h):= \left(\sum_{K\in\mathcal{T}_h}\eta_h^2(\lambda_h,u_h,K)\right)^{1/2}.
\end{equation}
\begin{figure}[ht]
\centering
\includegraphics[width=7cm,height=5.2cm]{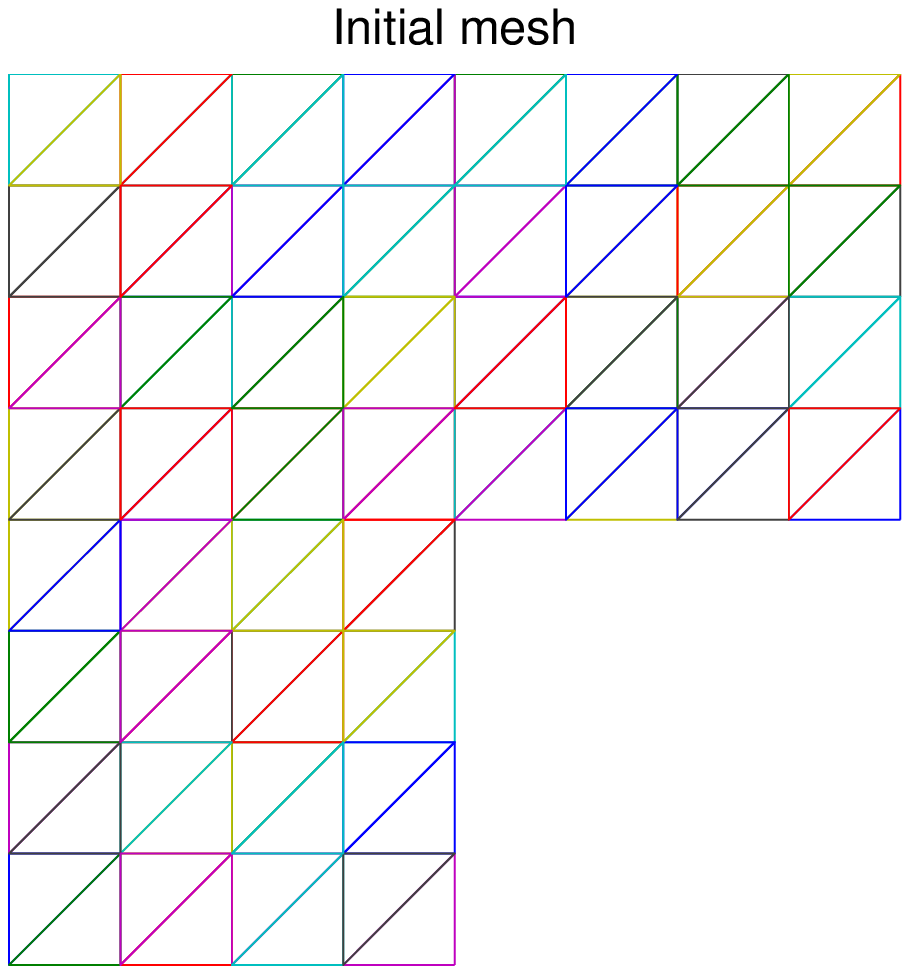}
\includegraphics[width=7cm,height=5.2cm]{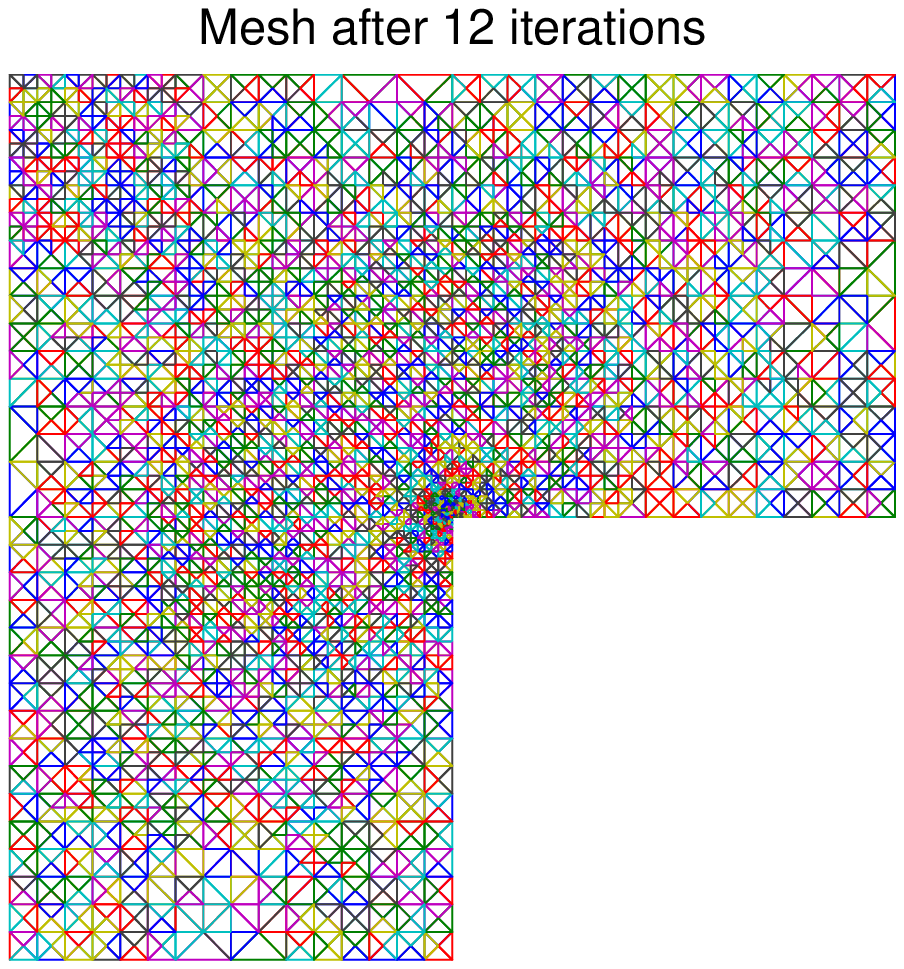}
\caption{\small\texttt The initial mesh of L-shape domain (left) and the triangulation after 12 adaptive
iterations for L-shape domain (right).}
\label{L_shape_mesh}
\end{figure}

In this example, we solve (\ref{GPEfem}) in the linear conforming finite element
space $V_h$ and solve the dual problem (\ref{Dual_Problem_fem}) in the finite element space $\bfW_h^1$
, respectively. Figure \ref{L_shape_mesh}  show the initial mesh (left) and corresponding
adaptive mesh (right) after 15 iterations. The corresponding
numerical results are presented in Figure \ref{L_shape_eigenpair} which shows that the a posteriori
error estimate $\eta(\lambda_h,u_h,\bfp_h^*)$ is also efficient even on the adaptive meshes when we
solve the dual problem in $\bfW_h^1$. Figure \ref{L_shape_eigenpair}
also shows the eigenvalue approximation $\lambda_h^L$ and ground state energy approximation
$E_h^L$ are really asymptotic lower bounds for the first eigenvalue $\lambda$ and ground state energy $E$, respectively.
\begin{figure}[ht]
\centering
\includegraphics[width=6cm,height=5.5cm]{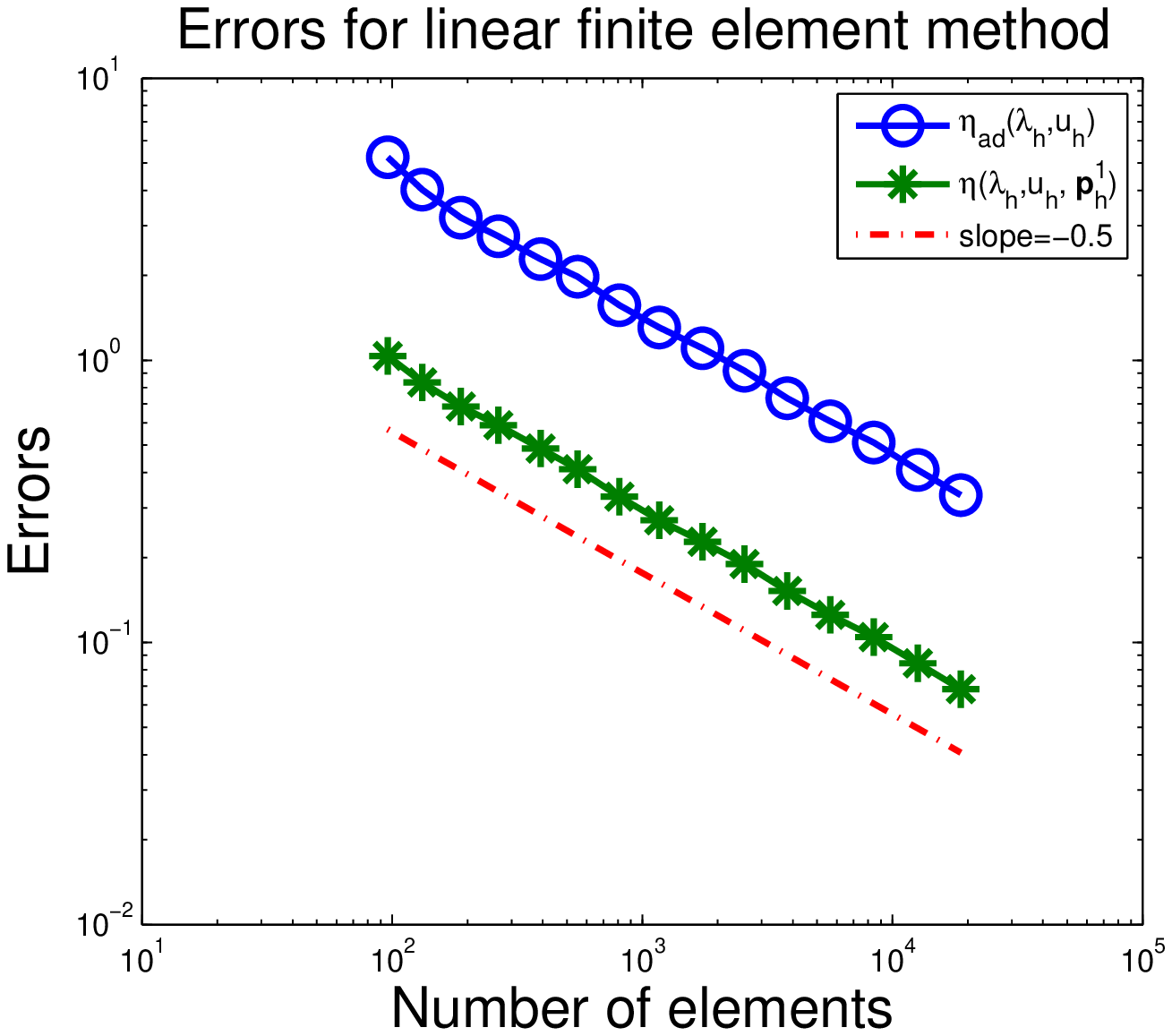}
\includegraphics[width=6cm,height=5.5cm]{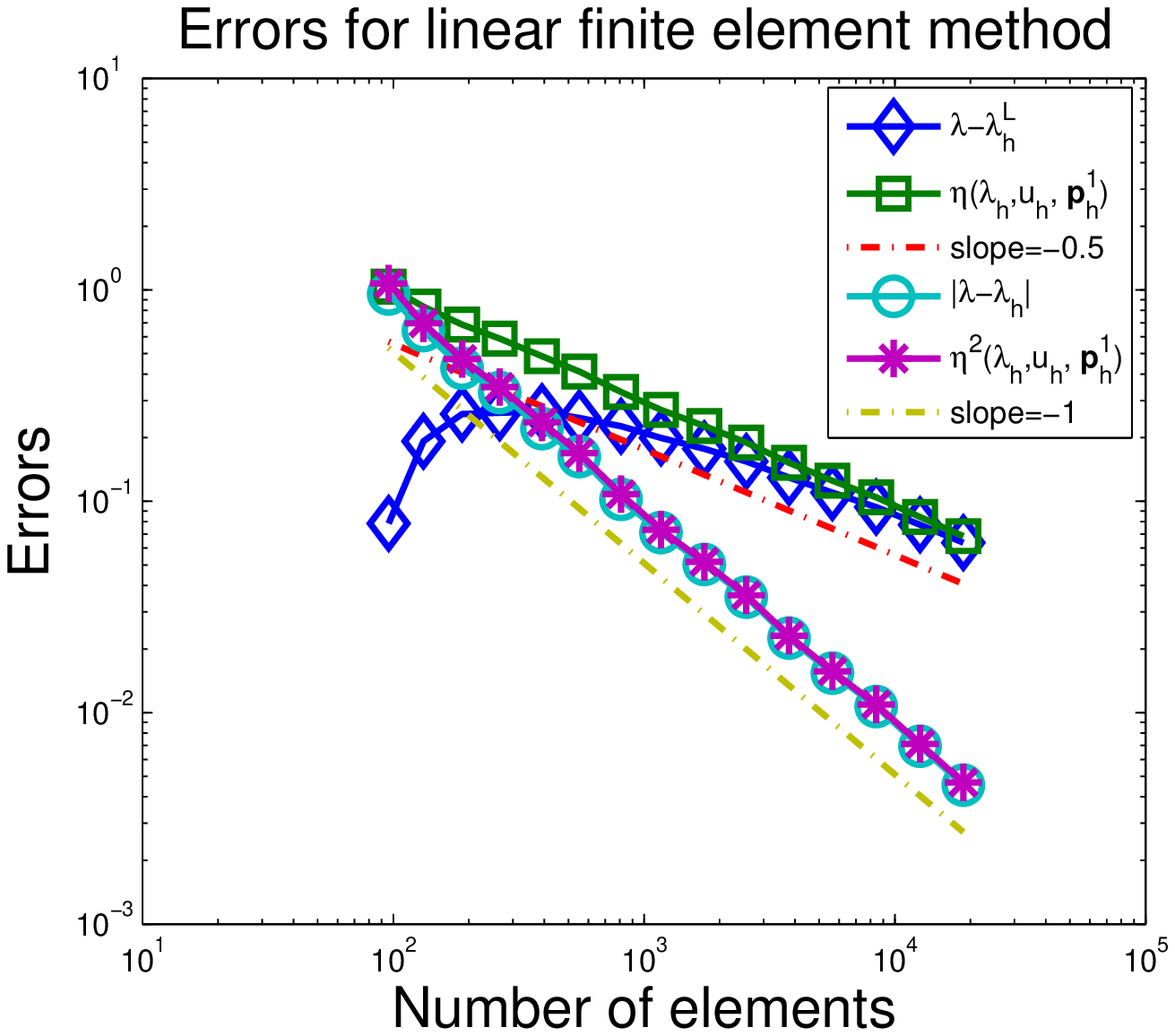}
\includegraphics[width=6cm,height=5.5cm]{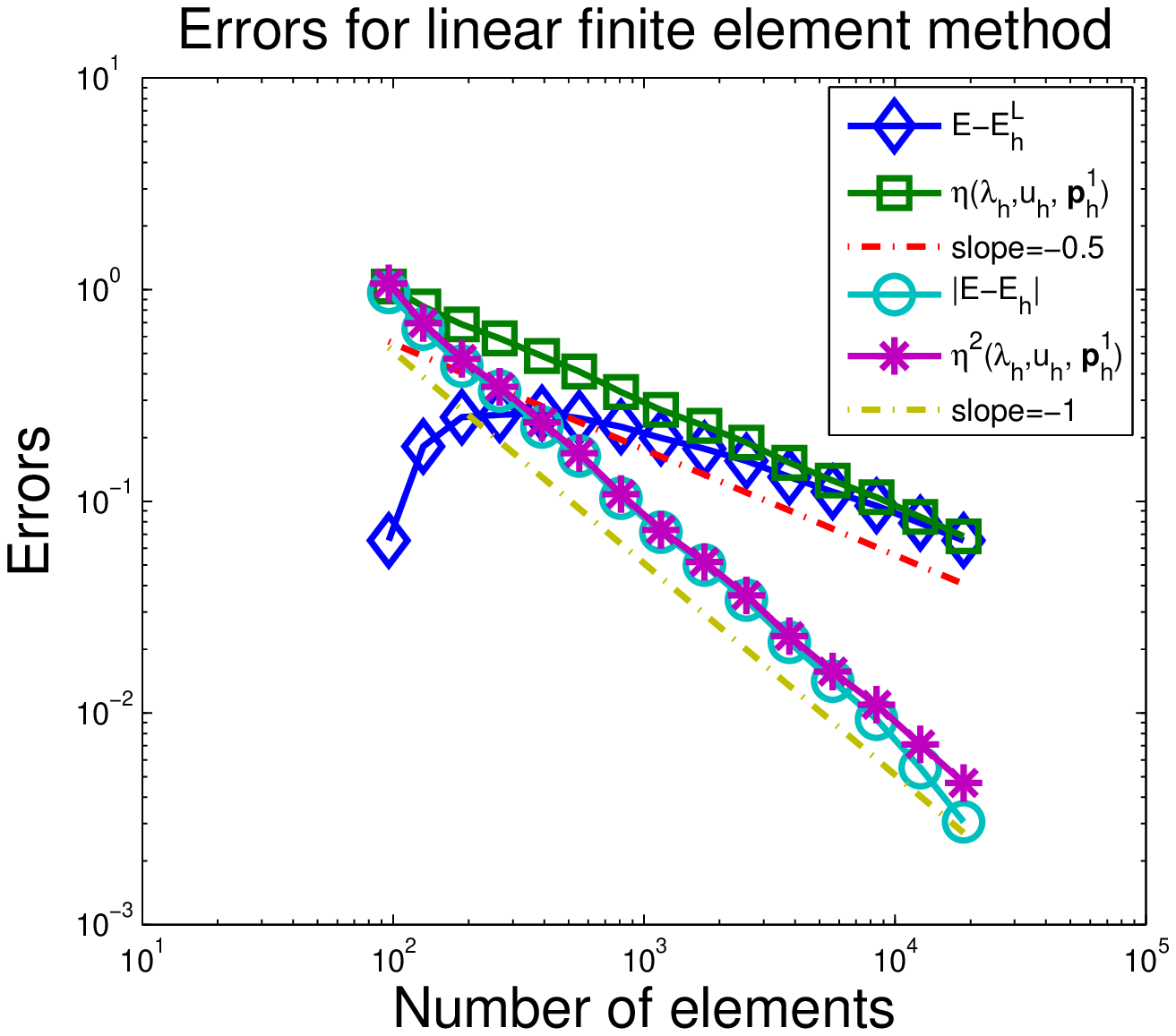}
\caption{\small\texttt The errors for the L shape domain when the eigenvalue problem is solved by the
linear finite element method, where $\eta(\lambda_h,u_h,\bfp_h^1)$ denotes the a posteriori
error estimates $\eta(\lambda_h,u_h,\bfp_h^*)$ when the dual problem is solved in $\bfW_h^1$,
and $\lambda_h^L$ denotes the asymptotic lower bounds of the first eigenvalue and $E_h^L$ denotes
the asymptotic lower bound of the ground state energy.}
\label{L_shape_eigenpair}
\end{figure}

\section{Concluding remarks}
In this paper, we give a computable error estimate for the eigenpair approximation by
the general conforming finite element methods on general meshes. Furthermore, the
asymptotic lower bound of the first eigenvalue and ground state energy can be obtained by the computable
error estimate. Some numerical examples
are provided to demonstrate the validation of the asymptotic lower
bounds for he first eigenvalue and ground state energy. The method here can be extended to many other
nonlinear eigenvalue problems, such as Kohn-Sham model for Schr\"{o}dinger equation. Moreover, we can 
adopt the efficient numerical methods to obtain these lower bound, such as multilevel correction
 and multigrid method (cf. \cite{LinXie,Xie_JCP,XieXie}). 
 
From the definitions (\ref{Asymptotic_lower_Bound}) and (\ref{Asymptotic_lower_Bound_Energy})
and numerical examples, we can find the accuracy of lower bounds $\lambda_h^L$ and $E_h^L$
is not optimal. How to produce the lower bounds with optimal accuracy will be our future work.

\end{document}